\documentclass[12pt,reqno]{amsart}
\usepackage[left=80pt,right=80pt]{geometry}
\usepackage[usenames]{color}
\usepackage{amsmath}
\usepackage{amssymb}
\usepackage{amsthm}
\usepackage{enumitem}
\usepackage{float}
\usepackage{graphicx}
\usepackage{tikz}
\usepackage{subcaption}
\usepackage{cases}
\usepackage{blkarray}
\usepackage{tikz-qtree}
\usetikzlibrary{graphs,graphs.standard,calc}
\usepackage{hyperref,url}
\hypersetup{
	colorlinks=true,
  linkcolor=black,          
  citecolor=black,         
  filecolor=black,      
  urlcolor=black           
}

\newtheorem{prop}{Proposition}
\newtheorem{lemma}[prop]{Lemma}
\newtheorem{theorem}[prop]{Theorem}

\theoremstyle{definition}
\newtheorem{definition}[prop]{Definition}
\newtheorem{remark}[prop]{Remark}
\newtheorem{example}[prop]{Example}

\newcommand{\seqnum}[1]{\href{https://oeis.org/#1}{\rm \underline{#1}}}
\allowdisplaybreaks
\newcommand{\mylabel}[2]{#2\def\@currentlabel{#2}\label{#1}}

\begin{document}
\tikzset{mystyle/.style={matrix of nodes,
        nodes in empty cells,
        row 1/.style={nodes={draw=none}},
        row sep=-\pgflinewidth,
        column sep=-\pgflinewidth,
        nodes={draw,minimum width=1cm,minimum height=1cm,anchor=center}}}
\tikzset{mystyleb/.style={matrix of nodes,
        nodes in empty cells,
        row sep=-\pgflinewidth,
        column sep=-\pgflinewidth,
        nodes={draw,minimum width=1cm,minimum height=1cm,anchor=center}}}

\title{Staircase graph words}

\author[SELA FRIED]{Sela Fried$^{\dagger}$}
\thanks{$^{\dagger}$ Department of Computer Science, Israel Academic College,
52275 Ramat Gan, Israel.
\\
\href{mailto:friedsela@gmail.com}{\tt friedsela@gmail.com}}
\author[TOUFIK MANSOUR]{Toufik Mansour$^{\sharp}$}
\thanks{$^{\sharp}$ Department of Mathematics, University of Haifa, 3103301 Haifa,
Israel.\\
\href{mailto:tmansour@univ.haifa.ac.il}{\tt tmansour@univ.haifa.ac.il}}

\maketitle

\begin{abstract}
Generalizing the notion of staircase words, introduced by Knopfmacher et.\ al, we define staircase graph words. These are functions $w$ from the vertex set $V$ of a graph into the set $\{1,2,\ldots,k\}$, such that $|w(x)-w(y)|\leq 1$, for every adjacent $x,y\in V$. We find the explicit generating functions for the number of staircase graph words for the grid graph, the rectangle-triangular graph and the king's graph, all of size $2\times n$.
\bigskip

\smallskip

\noindent
\textbf{Math.~Subj.~Class.:} 68R05, 05A05, 05A15.
\end{abstract}

\section{Introduction}
Let $k$ and $n$ be two positive integers and let $[k]=\{1,2,\ldots,k\}$ be an alphabet. A word over $[k]$ (of length $n$) is merely an element of $[k]^n$. Restricted words are words that do not contain certain subwords. Since the work of Burstein \cite{B}, that may be regarded as the first  systematic study of restricted words, much research has been devoted to the study of this subject.

In this work we concentrate on a specific kind of restricted words, that was introduced by Knopfmacher et al.\ \cite{K}, namely \emph{staircase words}. These are words $x=x_1\cdots x_n\in[k]^n$ such that $|x_i-x_{i+1}|\leq 1$, for every $1\leq i\leq n-1$. We propose the following generalization.

\begin{definition}\label{def;1}
Let $G$ be a graph with vertex set $V$. A \emph{$(G,k)$-word} is any function $w\colon V\to [k]$. A $(G,k)$-word $w$ is called \emph{staircase} if $|w(x)-w(y)|\leq 1$, for every adjacent $x,y\in V$. The number of staircase $(G,k)$-words is denoted by $s_k(G)$.
\end{definition}

\begin{example} A staircase word of length $n$ (in the sense of \cite{K}) is a $(P_n,k)$-staircase word (in the sense of Definition \ref{def;1}), where $P_n$ is the path graph of size $n$. Knopfmacher et al.\ have shown in \cite[Theorem 2.2]{K} that the generating function for the number of staircase words of length $n$ is given by $$1+\frac{x(k-(3k+2)x)}{(1-3x)^2}+\frac{2x^2}{(1-3x)^2}\frac{1+U_{k-1}\left(\frac{1-x}{2x}\right)}{U_k\left(\frac{1-x}{2x}\right)},$$ where $U_k(x)$ is the Chebyshev polynomial of the second kind (of degree $k$).

Moreover, Knopfmacher et al.\ \cite{K} have also considered \emph{staircase-cyclic words}. These are staircase words $x=x_1\cdots x_n$ such that $|x_1-x_n|\leq 1$. In our terminology, these are $(C_n,k)$-words, where $C_n$ is the cycle graph of size $n$. 
\end{example}

In this work we concentrate on the grid graph, the rectangle-triangular graph, and the king's graph, all of size $2\times n$, defined as follows (see Figure \ref{fig;1} below for a visualization).
\begin{definition}
Let $V=\{(i,j)\;:\;i=1,2 \textnormal{ and } 1\leq j\leq n\}$.
The \emph{grid graph of size $2\times n$}, denoted by $P_2\times P_n$, is the graph whose vertex set is $V$ and two vertices $(i_1,j_1),(i_2,j_2)\in V$ are adjacent if $|i_1-i_2|+|j_1-j_2|=1$.
The \emph{rectangle-triangular graph of size $2\times n$}, denoted by $RT_{2,n}$, is the graph whose vertex set is $V$ and two vertices $(i_1,j_1),(i_2,j_2)\in V$ are adjacent if (a) $|i_1-i_2|+|j_1-j_2|=1$, or (b) $i_1=1, i_2=2$ and $j_2 = j_1 +1$, or (c) $i_1=2, i_2=1$ and $j_2 = j_1 -1$.
The \emph{king's graph of size $2\times n$}, denoted by $KG_{2,n}$ (cf.~\cite[p.~223]{C}, is the graph whose vertex set is $V$ and two vertices $(i_1,j_1),(i_2,j_2)\in V$ are adjacent if (a) $|i_1-i_2|+|j_1-j_2|=1$, or (b) $|i_1-i_2|=|j_1-j_2|=1$.
\end{definition}

\begin{remark}
Let $G$ be one of the graphs defined above with the vertex set $V$. It will be convenient to think of $V$ as the index set of a $2\times n$ matrix. In particular, $(1,1)$ corresponds to the position of the upper left entry of the matrix. It follows that, in this work, every $(G,k)$-word corresponds to a $2\times n$ matrix, whose entries belong to $[k]$.  We shall make extensive use of the following refinement of $s_k(G)$: For $i,j\in[k]$, we denote by $s_k(G,i,j)$ the number of staircase $(G, k)$-words whose first column is $(i,j)^T$, where $v^T$ stands for the  transpose of the column vector $v$.
\end{remark}

\begin{figure}[H]
\centering
\begin{subfigure}{\textwidth}\centering
\scalebox{0.6}{
\begin{tikzpicture}[shorten >=0pt,node distance=3cm,auto]
\node[draw,circle,inner sep=0.25cm, fill=black!5] (1)  at (-4, 0) [thick] {};
\node[draw,circle,inner sep=0.25cm, fill=black!5] (2)  at (-4, -2) [thick] {};
\node[draw,circle,inner sep=0.25cm, fill=black!5] (3)  at (-2, 0) [thick] {};
\node[draw,circle,inner sep=0.25cm, fill=black!5] (4)  at (-2, -2) [thick] {};
\node[draw,circle,inner sep=0.25cm, fill=black!5] (5)  at (0, 0) [thick] {};
\node[draw,circle,inner sep=0.25cm, fill=black!5] (6)  at (0, -2) [thick] {};
\node[draw,circle,inner sep=0.25cm, fill=black!5] (7)  at (2, 0) [thick] {};
\node[draw,circle,inner sep=0.25cm, fill=black!5] (8)  at (2, -2) [thick] {};
\node[draw,circle,inner sep=0.25cm, fill=black!5] (9)  at (4, 0) [thick] {};
\node[draw,circle,inner sep=0.25cm, fill=black!5] (10) at (4, -2) [thick] {};

\path (1) edge[thick,-] node[above] {} (2);
\path (3) edge[thick,-] node[above] {} (4);
\path (5) edge[thick,-] node[above] {} (6);
\path (7) edge[thick,-] node[above] {} (8);
\path (9) edge[thick,-] node[above] {} (10);
\path (1) edge[thick,-] node[above] {} (3);
\path (3) edge[thick,-] node[above] {} (5);
\path (5) edge[thick,-] node[above] {} (7);
\path (7) edge[thick,-] node[above] {} (9);
\path (2) edge[thick,-] node[above] {} (4);
\path (4) edge[thick,-] node[above] {} (6);
\path (6) edge[thick,-] node[above] {} (8);
\path (8) edge[thick,-] node[above] {} (10);
\end{tikzpicture}}
\caption{The grid graph $P_2\times P_5.$}
\label{fig:1}
\end{subfigure}
\begin{subfigure}{\textwidth}\centering
\scalebox{0.6}{
\begin{tikzpicture}[shorten >=0pt,node distance=3cm,auto]
\node[draw,circle,inner sep=0.25cm, fill=black!5] (1)  at (-4, 0) [thick] {};
\node[draw,circle,inner sep=0.25cm, fill=black!5] (2)  at (-4, -2) [thick] {};
\node[draw,circle,inner sep=0.25cm, fill=black!5] (3)  at (-2, 0) [thick] {};
\node[draw,circle,inner sep=0.25cm, fill=black!5] (4)  at (-2, -2) [thick] {};
\node[draw,circle,inner sep=0.25cm, fill=black!5] (5)  at (0, 0) [thick] {};
\node[draw,circle,inner sep=0.25cm, fill=black!5] (6)  at (0, -2) [thick] {};
\node[draw,circle,inner sep=0.25cm, fill=black!5] (7)  at (2, 0) [thick] {};
\node[draw,circle,inner sep=0.25cm, fill=black!5] (8)  at (2, -2) [thick] {};
\node[draw,circle,inner sep=0.25cm, fill=black!5] (9)  at (4, 0) [thick] {};
\node[draw,circle,inner sep=0.25cm, fill=black!5] (10) at (4, -2) [thick] {};

\path (1) edge[thick,-] node[above] {} (2);
\path (3) edge[thick,-] node[above] {} (4);
\path (5) edge[thick,-] node[above] {} (6);
\path (7) edge[thick,-] node[above] {} (8);
\path (9) edge[thick,-] node[above] {} (10);
\path (1) edge[thick,-] node[above] {} (3);
\path (3) edge[thick,-] node[above] {} (5);
\path (5) edge[thick,-] node[above] {} (7);
\path (7) edge[thick,-] node[above] {} (9);
\path (2) edge[thick,-] node[above] {} (4);
\path (4) edge[thick,-] node[above] {} (6);
\path (6) edge[thick,-] node[above] {} (8);
\path (8) edge[thick,-] node[above] {} (10);

\path (1) edge[thick,-] node[above] {} (4);
\path (3) edge[thick,-] node[above] {} (6);
\path (5) edge[thick,-] node[above] {} (8);
\path (7) edge[thick,-] node[above] {} (10);
\end{tikzpicture}}
\caption{The rectangle-triangular graph $RT_{2,5}.$}
\label{fig;2}
\end{subfigure}
\begin{subfigure}{\textwidth}\centering
\scalebox{0.6}{
\begin{tikzpicture}[shorten >=0pt,node distance=3cm,auto]
\node[draw,circle,inner sep=0.25cm, fill=black!5] (1)  at (-4, 0) [thick] {};
\node[draw,circle,inner sep=0.25cm, fill=black!5] (2)  at (-4, -2) [thick] {};
\node[draw,circle,inner sep=0.25cm, fill=black!5] (3)  at (-2, 0) [thick] {};
\node[draw,circle,inner sep=0.25cm, fill=black!5] (4)  at (-2, -2) [thick] {};
\node[draw,circle,inner sep=0.25cm, fill=black!5] (5)  at (0, 0) [thick] {};
\node[draw,circle,inner sep=0.25cm, fill=black!5] (6)  at (0, -2) [thick] {};
\node[draw,circle,inner sep=0.25cm, fill=black!5] (7)  at (2, 0) [thick] {};
\node[draw,circle,inner sep=0.25cm, fill=black!5] (8)  at (2, -2) [thick] {};
\node[draw,circle,inner sep=0.25cm, fill=black!5] (9)  at (4, 0) [thick] {};
\node[draw,circle,inner sep=0.25cm, fill=black!5] (10) at (4, -2) [thick] {};

\path (1) edge[thick,-] node[above] {} (2);
\path (3) edge[thick,-] node[above] {} (4);
\path (5) edge[thick,-] node[above] {} (6);
\path (7) edge[thick,-] node[above] {} (8);
\path (9) edge[thick,-] node[above] {} (10);
\path (1) edge[thick,-] node[above] {} (3);
\path (3) edge[thick,-] node[above] {} (5);
\path (5) edge[thick,-] node[above] {} (7);
\path (7) edge[thick,-] node[above] {} (9);
\path (2) edge[thick,-] node[above] {} (4);
\path (4) edge[thick,-] node[above] {} (6);
\path (6) edge[thick,-] node[above] {} (8);
\path (8) edge[thick,-] node[above] {} (10);

\path (1) edge[thick,-] node[above] {} (4);
\path (3) edge[thick,-] node[above] {} (6);
\path (5) edge[thick,-] node[above] {} (8);
\path (7) edge[thick,-] node[above] {} (10);

\path (2) edge[thick,-] node[above] {} (3);
\path (4) edge[thick,-] node[above] {} (5);
\path (6) edge[thick,-] node[above] {} (7);
\path (8) edge[thick,-] node[above] {} (9);
\end{tikzpicture}}
\caption{The king's graph $KG_{2,5}.$}
\label{fig:3}
\end{subfigure}
\caption{Examples of the three graph families considered in this work.}\label{fig;1}
\end{figure}
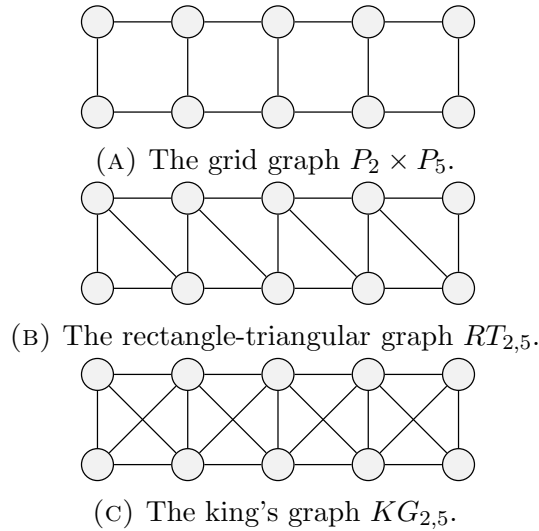 \noindent

\begin{example}
Table \ref{table} below shows the numbers of staircase $(G_n,3)$-words for $n=1,2,\ldots,7$, where $G_n$ is either $P_2\times P_n, RT_{2,n}$ or $KG_{2,n}$. The last column refers to the On-Line Encyclopedia of Integer Sequences (OEIS) \cite{OL}.
\begin{table}[H]
\begin{center}
\begin{tabular}{||c| c |c|c|c|c|c|c|c||}
 \hline
 $n$ & $1$& $2$& $3$& $4$ & $5$&  $6$ & $7$ & OEIS \\ [0.5ex]
 \hline\hline
 $s_3(P_2\times P_n)$  & $7$&$35$&$181$&$933$&$4811$&$24807$&$127913$& \seqnum{A051926} \\
\hline
$s_3(RT_{2,n})$ &$7$& $33$& $161$& $783$& $3809$& $18529$& $90135$ & not registered \\
 \hline
 $s_3(KG_{2,n})$ & $7$& $31$& $145$& $673$& $3127$& $14527$& $67489$ & \seqnum{A086901} \\
    \hline
\end{tabular}
\caption{Number of staircase graph-words corresponding to the graphs considered in this work, over an alphabet of size $3$.}\label{table}
\end{center}
\end{table}
\end{example}

\section{Main results}
We apply the kernel method (e.g., \cite{P}). The order in which the graphs are studied is according to the complexity of the analysis, beginning with the easiest graph, namely, the king's graph.

\subsection{\texorpdfstring{The king's graph}{}}
Let $S_k(x)=\sum_{n\geq1}s_k(KG_{2,n})x^n$ be the generating function of the number of staircase $(KG_{2,n}, k)$-words and, for $i,j\in[k]$, we define  $S_k(x,i,j)=\sum_{n\geq1}s_k(KG_{2,n},i,j)x^n$ to be the generating function of the number of staircase $(KG_{2,n}, k)$-words whose first column is $(i,j)^T$. We set $S_k(x,i,j)=0$ if either $i\notin [k]$ or $j\notin [k]$.

\begin{lemma}\label{lem;1a}
\begin{enumerate}
\item  We have
\begin{align}
S_k(x)=\sum_{i=1}^kS_k(x,i,i)+2\sum_{i=1}^{k-1}S_k(x,i+1,i).\label{eq1a}
\end{align}
    \item The generating function $S_k(x,i,i)$ satisfies
\begin{align}
S_k(x,i,i)&=x+xS_k(x,i-1,i-1)+2xS_k(x,i,i-1)\nonumber\\
&+xS_k(x,i,i)+2xS_k(x,i+1,i)+xS_k(x,i+1,i+1).\label{eq;1a}
\end{align}
\item The generating function $S_k(x,i+1,i)$ satisfies
\begin{align}
S_k(x,i+1,i)&=x+xS_k(x,i,i)+2xS_k(x,i+1,i)+xS_k(x,i+1,i+1).\label{eq;5a}
\end{align}
\end{enumerate}
\end{lemma}

\begin{proof}
\begin{enumerate}
\item This is clear, since $s_k(KG_{2,n},i,j)=s_k(KG_{2,n},j,i)$.
\item Let $w$ be a staircase $KG_{2,n}$-word whose first column is $(i,i)^T$. Either $n=1$, or $n\geq2$. In the latter case, the second column of $w$ must be one of
$$(i-1,i-1)^T,(i,i-1)^T,(i-1,i)^T,(i,i)^T,(i+1,i)^T,(i,i+1)^T,(i+1,i+1)^T.$$ Writing this in terms of generating functions, we obtain (\ref{eq;1a}).
\item Similar to the previous case, only, this time, the second column of $w$ must be one of
\[
(i,i)^T,(i+1,i)^T,(i,i+1)^T,(i+1,i+1)^T.\qedhere\]
\end{enumerate}
\end{proof}

Equation (\ref{eq1a}) motivates the definition of two additional generating functions (in the variables $x$ and $t$):
$$A_k(x,t)=\sum_{i=1}^kS_k(x,i,i)t^{i-1}\mbox{ and }B_k(x,t)=\sum_{i=1}^{k-1}S_k(x,i+1,i)t^{i-1}.$$

\begin{lemma}\label{lem;3a}
We have
\begin{align}
A_k(x,1)&=\frac{x(2A_{k}(x,0)x+2S_{k}(x,k,k)x-2kx+A_{k}(x,0)+S_{k}(x,k,k)-k+4x)}{2x^{2}+5x-1},\label{eq;9a}\\
B_k(x,1)&=-\frac{x(A_{k}(x,0)x+S_{k}(x,k,k)x-kx-A_{k}(x,0)-S_{k}(x,k,k)+k+3x-1)}{2x^{2}+5x-1}.\label{eq;9b}
\end{align}
Furthermore,
\begin{equation}\label{eq;gz}
    S_k(x) = \frac{x(3A_{k}(x,0)+3S_{k}(x,k,k)-3k-2x+2)}{2x^{2}+5x-1}
\end{equation}
\end{lemma}

\begin{proof}
Multiplying (\ref{eq;1a}) by $t^{i-1}$ and summing over $i=1,2,\ldots,k$, we obtain
\begin{align}
A_k(x,t)&=\frac{x(1-t^k)}{1-t}+xt(A_k(x,t)-S_k(x,k,k)t^{k-1})+2xtB_k(x,t)\nonumber\\
&+xA_k(x,t)+2xB_k(x,t)+\frac{x}{t}(A_k(x,t)-A_k(x,0))\label{eq;20a}.
\end{align} Similarly, multiplying (\ref{eq;5a}) by $t^{i-1}$ and summing over $i=1,2,\ldots,k-1$, we obtain
\begin{align}
B_{k}(x,t)&=\frac{x(1-t^{k-1})}{1-t}+x\left(A_{k}(x,t)-S_{k}(x,k,k)t^{k-1}\right)\nonumber\\&+2xB_{k}(x,t)+\frac{x}{t}\left(A_{k}(x,t)-A_{k}(x,0)\right)\label{eq;25a}.
\end{align} Taking $\lim_{t\to 1}$ in \eqref{eq;20a} and \eqref{eq;25a}, we obtain a linear system of two equations in the variables $A_k(x,1)$ and $B_k(x,1)$. Solving this system gives \eqref{eq;9a} and \eqref{eq;9b}.

To obtain \eqref{eq;gz}, notice that equation \eqref{eq1a} is equivalent to $S_k(x) = A_k(x,1) + 2B_k(x,1)$. The assertion follows now from the previous part.
\end{proof}

\begin{theorem}\label{mth1}
We have
\begin{align*}
S_k(x)&=\frac{x(t_{1}+2x)(3kt_{1}+2t_{1}x-3k-2t_{1}-2x-4)t_{1}^{k}}{(1-5x-2x^{2})(t_1^{k+1}+2t_1^{k}x+2t_1x+1)(t_1-1)}\\
&\qquad\qquad+\frac{(2t_{1}x+1)(3kt_{1}+2t_{1}x-3k+4t_{1}-2x+2)}{(1-5x-2x^{2})(t_1^{k+1}+2t_1^{k}x+2t_1x+1)(t_1-1)},
\end{align*} where $t_1 = \frac{1-3x-2x^2-\sqrt{(1-3x-2x^2)^{2}-4x^2}}{2x}$.
\end{theorem}

\begin{proof}
First, we solve (\ref{eq;25a}) for $B_k(x,t)$ and substitute it into (\ref{eq;20a}). This gives us the equation
\begin{align}
\frac{K(t)}{t}A_{k}(x,t)&=\frac{(2tx+1)x}{t}A_{k}(x,0)+(t+2x)xt^{k-1}S_{k}(x,k,k)\nonumber\\&+\frac{2x^{2}(1-t^{k-1})(t+1)+(x-2x^{2})(1-t^{k})}{t-1},\label{eq;77a}
\end{align} where $K(t)=xt^{2}+(2x^{2}+3x-1)t+x$. The two roots of the kernel equation $K(t)=0$ are given by $t_1$ and $1/t_1$.
By substituting $t_1$ and $1/t_1$ into \eqref{eq;77a}, we obtain a linear system of two equations in the variables $A_k(x,0)$ and $S_k(x,k,k)$. Solving this system we obtain
\begin{align}
A_k(x,0) &= S_k(x,k,k)\\ 
&=\frac{4\left(t_1^{2k}-t_1^{k+1}(1+t_1)+t_1^3\right)}{T}x^{2}\nonumber\\ &+\frac{2\left(2t_1^{2k+1}-t_1^{k}(1+t_1)(1+t_1^2)+2t_1^2\right)}{T}x
+\frac{t_1^{2k+2}-t_1^{k+1}(1+t_1)+t_1}{T},  \nonumber
\end{align} where $T = (1-t_1)\left(1-t_1^{2k+2}-4(t_1^{2k+1}-t_1)x-4(t_1^{2k}-t_1^2)x^{2}\right)$.
Substituting these in \eqref{eq;gz}, the assertion follows.
\end{proof}

\begin{example}
    The generating function of the number of staircase $(KG_{2,n}, 3)$-words is given by $$\frac{x(3x+7)}{1-4x+3x^2},$$
that is, the number of staircase $(KG_{2,n}, 3)$-words is given by
$$\frac{5\sqrt{7}+7}{14}(2+\sqrt{7})^n-\frac{5\sqrt{7}-7}{14}(2-\sqrt{7})^n.$$
Also, the generating function of the number of staircase $(KG_{2,n}, 4)$-words is given by $$\frac{2x(5-12x-3x^2)}{1-7x+9x^2+6x^3}$$
and the generating function of the number of staircase $(KG_{2,n}, 5)$-words is given by
$$\frac{x(13-30x-42x^2-6x^3)}{1-7x+6x^2+18x^3-6x^4}.$$
\end{example}

\subsection{The grid graph}

Let $S_k(x)=\sum_{n\geq1}s_k(P_2\times P_n)x^n$ be the generating function of the number of staircase $(P_2\times P_n, k)$-words and, for $i,j\in[k]$, we define $S_k(x,i,j)=\sum_{n\geq1}s_k(P_2\times P_n,i,j)x^n$ to be the generating function of the number of staircase $(P_2\times P_n, k)$-words whose first column is $(i,j)^T$. We set $S_k(x,i,j)=0$ if either $i\notin [k]$ or $j\notin [k]$.

\begin{lemma}\label{lem;80}
\begin{enumerate}
\item  We have
\begin{align}
S_k(x)=\sum_{i=1}^kS_k(x,i,i)+2\sum_{i=1}^{k-1}S_k(x,i+1,i).\label{eq1}
\end{align}
\item The generating function $S_k(x,i,i)$ satisfies
\begin{align}
S_k(x,i,i)&=x+xS_k(x,i-1,i-1)+2xS_k(x,i,i-1)\nonumber\\
&+xS_k(x,i,i)+2xS_k(x,i+1,i)+xS_k(x,i+1,i+1).\label{eq;1}
\end{align}
\item The generating function $S_k(x,i+1,i)$ satisfies
\begin{align}
S_k(x,i+1,i)&=x+xS_k(x,i,i-1)+xS_k(x,i,i)\nonumber\\
&+2xS_k(x,i+1,i)+xS_k(x,i+1,i+1)+xS_k(x,i+2,i+1).\label{eq;5}
\end{align}
\end{enumerate}
\end{lemma}

\begin{proof}
\begin{enumerate}
\item This is clear.
\item Let $w$ be a staircase $P_2\times P_n$-word whose first column is $(i,i)^T$. Either $n=1$, or $n\geq2$. In the latter case, the second column of $w$ must be one of
$$(i-1,i-1)^T,(i,i-1)^T,(i-1,i)^T,(i,i)^T,(i+1,i)^T,(i,i+1)^T,(i+1,i+1)^T.$$ Writing this in terms of generating functions, we obtain (\ref{eq;1}).
\item Similar to the previous case, only, this time, the second column of $w$ must be one of
\[(i,i-1)^T,(i,i)^T,(i+1,i)^T,(i,i+1)^T,(i+1,i+1)^T,(i+2,i+1)^T.\qedhere\]
\end{enumerate}
\end{proof}

Equation (\ref{eq1}) motivates the definition of two additional generating functions (in the variables $x$ and $t$):
$$A_k(x,t)=\sum_{i=1}^kS_k(x,i,i)t^{i-1}\mbox{ and }B_k(x,t)=\sum_{i=1}^{k-1}S_k(x,i+1,i)t^{i-1}.$$

\begin{lemma}\label{lem;3c}
We have
\begin{align}
A_k(x,1)&=\frac{x(A_{k}(x,0)+4B_{k}(x,0)x
+4S_{k}(x,k,k-1)x+S_{k}(x,k,k)-k+4x)}{4x^{2}-7x+1}\label{eq;90a}\\
B_k(x,1)&=\frac{x((x-1)A_{k}(x,0)+(3x-1)B_{k}(x,0)+(3x-1)S_{k}(x,k,k-1))}{4x^{2}-7x+1}\nonumber\\&+\frac{x((x-1)S_{k}(x,k,k)+(3-k)x+k-1)}{4x^{2}-7x+1},\label{eq;90b}
\end{align}
Furthermore,
\begin{align}
    S_k(x) &= \frac{x((2x-3)A_{k}(x,0)+(2x-2)B_{k}(x,0)x+(2x-2)S_{k}(x,k,k-1))}{4x^{2}-7x+1}\nonumber\\&+\frac{x((2x-3)S_{k}(x,k,k)-k(2x-3)+2x-2)}{4x^{2}-7x+1}.\label{eq;lak}
\end{align}
\end{lemma}

\begin{proof}
Multiplying \eqref{eq;1} by $t^{i-1}$ and summing over $i=1,2,\ldots,k$, we obtain
\begin{align}
A_k(x,t)&=\frac{x(1-t^k)}{1-t}+xt(A_k(x,t)-S_k(x,k,k)t^{k-1})+2xtB_k(x,t)\nonumber\\
&+xA_k(x,t)+2xB_k(x,t)+\frac{x}{t}(A_k(x,t)-A_k(x,0))\label{eq;20}
\end{align} Similarly, multiplying \eqref{eq;5} by $t^{i-1}$ and summing over $i=1,2,\ldots,k-1$, we obtain
\begin{align}
B_k(x,t)&=\frac{x(1-t^{k-1})}{1-t}+xt(B_k(x,t)-S_k(x,k,k-1)t^{k-2})+x(A_k(x,t)-S_k(x,k,k)t^{k-1})\nonumber\\
&+2xB_k(x,t)+\frac{x}{t}(A_k(x,t)+B_k(x,t)-A_k(x,0)-B_k(x,0)).\label{eq;25}
\end{align}

Taking $\lim_{t\to 1}$ in \eqref{eq;20} and \eqref{eq;25}, we obtain a linear system of two equations in the variables $A_k(x,1)$ and $,B_k(x,1)$. Solving this system gives \eqref{eq;90a} and \eqref{eq;90b}.

To obtain \eqref{eq;lak}, notice that equation \eqref{eq1} is equivalent to $S_k(x) = A_k(x,1) + 2B_k(x,1)$ and the assertion follows from the previous part.
\end{proof}

\begin{theorem}\label{mth3}
We have
\begin{align*}
S_k(x)&=\frac{a_{1}(x,t_1,t_2)t_{1}^{k}t_{2}^{k}+a_{2}(x,t_1,t_2)t_{1}^{k}-a_{2}(x,t_2,t_1)t_{2}^{k}+a_{3}(x,t_1,t_2)}{(t_1-1)(t_2-1)(4x^2 - 7x + 1)(b_1(x,t_1,t_2)t_1^kt_2^k + b_2(x,t_1,t_2)t_1^k -b_2(x,t_2,t_1)t_2^k + b_3(x,t_1,t_2))},
\end{align*} where

\begin{align}
a_1(x,t_1,t_2)&=t_1t_2(t_2-t_1)(2k(t_2^2 - 1)(t_1^2-1)x^3\nonumber\\& + (-3kt_1^2t_2^2 + 3kt_1^2 - 2kt_1t_2 + 3kt_2^2 - 2t_1^2t_2 - 2t_1t_2^2 + 2kt_1 + 2kt_2 - 2t_1t_2 - 5k - 2)x^2\nonumber\\& + (3kt_1t_2 + 2t_1^2t_2 + 2t_1t_2^2 - 3kt_1 - 3kt_2 + 6t_1t_2 + 3k + 4)x - 2t_2t_1),\nonumber\\
a_2(x,t_1,t_2)&=t_1(1-t_1t_2)(2k(t_1-1)(t_2 + 1)(t_2-1)(t_1 + 1)x^3\nonumber\\& + (-3kt_1^2t_2^2 - 2kt_1t_2^2 + 3kt_1^2 + 2kt_1t_2 + 5kt_2^2 + 2t_1^2t_2 - 2kt_2 + 2t_1t_2 + 2t_2^2 - 3k + 2t_1)x^2\nonumber\\& + (3kt_1t_2^2 - 3kt_1t_2 - 3kt_2^2 - 2t_1^2t_2 + 3kt_2 - 6t_1t_2 - 4t_2^2 - 2t_1)x + 2t_2t_1),\nonumber\\
a_3(x,t_1,t_2)&=(t_1-t_2)(2k(t_1^2-1)(t_2^2-1)x^3\nonumber\\& + (-5kt_1^2t_2^2 + 2kt_1^2t_2 + 2kt_1t_2^2 - 2t_1^2t_2^2 + 3kt_1^2 - 2kt_1t_2 + 3kt_2^2 - 2t_1t_2 - 3k - 2t_1 - 2t_2)x^2\nonumber\\& + (3kt_1^2t_2^2 - 3kt_1^2t_2 - 3kt_1t_2^2 + 4t_1^2t_2^2 + 3kt_1t_2 + 6t_1t_2 + 2t_1 + 2t_2)x - 2t_1t_2),\nonumber\\
b_1(x,t_1,t_2)&= t_1t_2(t_1-t_2)((t_1+1)(t_2+1)x-1),\nonumber\\
b_2(x,t_1,t_2)&= t_1(t_1t_2 - 1)((t_1+1)(t_2+1)x- t_2),\nonumber\\
b_3(x,t_1,t_2)&= (t_2-t_1)((t_1+1)(t_2+1)x - t_1t_2),\nonumber
\end{align} and
$$t_{1,2}=\frac{2-x\pm\sqrt{x(9x+8)}+\sqrt{\left(2-x\pm\sqrt{x(9x+8)}\right)^2-16x^2}}{4x}.$$
\end{theorem}

\begin{proof}
First, we solve \eqref{eq;25} for $B_k(x,t)$ and substitute it into \eqref{eq;20}. This gives us the equation
\begin{align}
\frac{K(t)}{t}A_k(x,t)&=-\frac{x(xt^2+t-x)}{t}A_k(x,0)-2x^2(t+1)B_k(x,0)
+x(xt^2-t-x)t^kS_k(x,k,k)
\nonumber\\&-2x^2(t+1)t^kS_k(x,k,k-1)+\frac{x(xt^{k+2}-t^{k+1}-xt^k+xt^2+t-x)}{1-t}.\label{eq;77}
\end{align} where $K(t)=x^2t^4+x(x-2)t^3+(1-3x)t^2+x(x-2)t+x^2$. The four roots of the kernel equation $K(t)=0$ are given by $t_1,1/t_1,t_2,$ and $1/t_2$. By substituting these four  roots into \eqref{eq;77}, we obtain a linear system of four equations in the variables $A_k(x,0), B_k(x,0), S_k(x,k,k-1)$, and $S_k(x,k,k)$. Solving this system we obtain
\begin{align}
&A_k(x,0) = S_k(x,k,k) =\nonumber\\ &\frac{f_1(x,t_1, t_2)t_1^kt_2^k+f_2(x,t_1, t_2)t_1^k-f_2(x,t_2, t_1)t_2^k+f_3(x,t_1, t_2)}{g_1(x,t_1,t_2)t_1^kt_2^k+g_2(x,t_1,t_2)t_1^k-g_2(x,t_2,t_1)t_2^k+g_3(x,t_1,t_2)},\nonumber\\
&B_k(x,0) = S_k(x,k-1,k) =\nonumber\\
&\frac{q_1(x,t_1, t_2)t_1^kt_2^k +q_2(x,t_1, t_2)t_1^k -q_2(x,t_2, t_1)t_2^k +q_3(x,t_1, t_2)}{2x\left(g_1(x,t_1,t_2)t_1^kt_2^k+g_2(x,t_1,t_2)t_1^k-g_2(x,t_2,t_1)t_2^k+g_3(x,t_1,t_2)\right)},
\end{align} where
\begin{align}
f_1(x,t_1, t_2) &= t_1t_2( t_1-t_2)(t_1t_2 + 1),\nonumber\\
f_2(x,t_1, t_2) &= t_1t_2(t_1 + t_2)(1-t_1t_2),\nonumber\\
f_3(x,t_1, t_2) &=t_1t_2(t_1t_2 + 1)(t_2-t_1),\nonumber\\
g_1(x,t_1,t_2) &= t_2t_1(t_1-1)(t_2-1)(t_1-t_2)(t_1t_2x + t_1x + t_2x + x - 1),\nonumber\\
g_2(x,t_1,t_2) &= t_1(t_1-1)(t_2-1)(t_1t_2 - 1)(t_1t_2x + t_1x + t_2x - t_2 + x),\nonumber\\g_3(x,t_1,t_2)&=(t_1-1)(t_2-1)(t_2-t_1)(t_1t_2x - t_1t_2 + t_1x + t_2x + x),\nonumber\\
q_1(x,t_1,t_2)&=t_1t_2(t_2 - t_1)(t_2^2x - t_2 - x)(t_1^2x - t_1 - x), \nonumber\\
q_2(x,t_1,t_2)&=t_1(1-t_1t_2)(t_2^2x + t_2 - x)(t_1^2x - t_1 - x), \nonumber\\
q_3(x,t_1,t_2)&=(t_1 - t_2)(t_2^2x + t_2 - x)(t_1^2x + t_1 - x). \nonumber
\end{align} Substituting these in \eqref{eq;lak}, the assertion follows.
\end{proof}

\begin{example}
    The generating function of the number of staircase $(P_2\times P_n, 3)$-words is given by $$\frac{x(7-x^2)}{1-5x-x^2+x^3}.$$
\end{example}

\subsection{The rectangle-triangular graph}
Let $S_k(x)=\sum_{n\geq1}s_k(RT_{2,n})x^n$ be the generating function of the number of staircase $(RT_{2,n}, k)$-words and, for $i,j\in[k]$, we define   $S_k(x,i,j)=\sum_{n\geq1}s_k(RT_{2,n},i,j)x^n$ to be the generating function of the number of staircase $(RT_{2,n}, k)$-words whose first column is $(i,j)^T$. We set $S_k(x,i,j)=0$ if either $i\notin [k]$ or $j\notin [k]$.

\begin{lemma}\label{lem;1b}
\begin{enumerate}
\item  We have
\begin{align}
S_k(x)=\sum_{i=1}^kS_k(x,i,i)+\sum_{i=1}^{k-1}S_k(x,i+1,i)+\sum_{i=1}^{k-1}S_k(x,i,i+1).\label{eq1b}
\end{align}
    \item The generating function $S_k(x,i,i)$ satisfies
\begin{align}
S_{k}(x,i,i)&=x+xS_{k}(x,i-1,i-1)+xS_{k}(x,i,i-1)+xS_{k}(x,i-1,i)+xS_{k}(x,i,i)\nonumber\\ &+xS_{k}(x,i+1,i)+xS_{k}(x,i,i+1)+xS_{k}(x,i+1,i+1).\label{eq;1b}
\end{align}
\item The generating function $S_k(x,i+1,i)$ satisfies
\begin{align}
S_k(x,i+1,i)&=x+xS_{k}(x,i,i)+xS_{k}(x,i+1,i)+xS_{k}(x,i,i+1)\nonumber\\&+xS_{k}(x,i+1,i+1)+xS_{k}(x,i+2,i+1).\label{eq;5b}
\end{align}
\item The generating function $S_k(x,i,i+1)$ satisfies
\begin{align}
S_{k}(x,i,i+1)&=x+xS_{k}(x,i-1,i)+xS_{k}(x,i,i)+xS_{k}(x,i+1,i)\nonumber\\&+xS_{k}(x,i,i+1)+xS_{k}(x,i+1,i+1).\label{eq;5c}
\end{align}
\end{enumerate}
\end{lemma}

\begin{proof}
\begin{enumerate}
\item This is clear.
\item Let $w$ be a staircase $(RT_{2,n}, k)$-word whose first column is $(i,i)^T$. Either $n=1$, or $n\geq2$. In the latter case, the second column of $w$ must be one of
$$(i-1,i-1)^T,(i,i-1)^T,(i-1,i)^T,(i,i)^T,(i+1,i)^T,(i,i+1)^T,(i+1,i+1)^T.$$ Writing this in terms of generating functions, we obtain (\ref{eq;1b}).
\item Similar to the previous case, only, this time, the second column of $w$ must be one of
$$(i,i)^T,(i+1,i)^T,(i,i+1)^T,(i+1,i+1)^T,(i+2,i+1)^T.$$
\item Similar to the previous cases, only, this time, the second column of $w$ must be one of
\[(i-1,i)^T,(i,i)^T,(i+1,i)^T,(i,i+1)^T, (i+1,i+1)^T.\qedhere\]
\end{enumerate}
\end{proof}

Equation (\ref{eq1b}) motivates the definition of three additional generating functions (in the variables $x$ and $t$):
\begin{align}
A_k(x,t)&=\sum_{i=1}^kS_k(x,i,i)t^{i-1},&&B_k(x,t)=\sum_{i=1}^{k-1}S_k(x,i+1,i)t^{i-1},\nonumber\\
C_k(x,t)&=\sum_{i=1}^{k-1}S_k(x,i,i+1)t^{i-1}.\nonumber
\end{align}

\begin{lemma}\label{lem;3b}
We have
\begin{align}
A_k(x,1)&=-\frac{x\left((x+1)A_{k}(x,0)+2xB_{k}(x,0)+2xS_{k}(x,k-1,k)\right)}{x^{2}-6x+1}\nonumber\\&-\frac{x\left((x+1)S_{k}(x,k,k)+(4-k)x-k\right)}{x^{2}-6x+1}
\label{eq;10a}\\
B_k(x,1)&=\frac{x\left((x-1)^{2}A_{k}(x,0)+(2x^{2}-5x+1)B_{k}(x,0)+x(x+1)S_{k}(x,k-1,k)\right)}{x^{3}-7x^{2}+7x-1}\nonumber\\&+\frac{x\left((x-1)^{2}S_{k}(x,k,k)+(3-k)x^{2}+2x(k-2)-k+1\right)}{x^{3}-7x^{2}+7x-1},\label{eq;10b}\\
C_k(x,1)&=\frac{x\left((x-1)^{2}A_{k}(x,0)+x(x+1)B_{k}(x,0)+(2x^{2}-5x+1)S_{k}(x,k-1,k)\right)}{x^{3}-7x^{2}+7x-1}\nonumber\\&+\frac{x\left((x-1)^{2}S_{k}(x,k,k)+(3-k)x^{2}+2x(k-2)-k+1\right)}{x^{3}-7x^{2}+7x-1}.\label{eq;10c}
\end{align}
Furthermore,
\begin{align}
    S_k(x) &= \frac{x\left((x-3)A_{k}(x,0)+(x-1)B_{k}(x,0)+(x-1)S_{k}(x,k-1,k)\right)}{x^{2}-6x+1}\nonumber\\&+\frac{x\left((x-3)S_{k}(x,k,k)-kx+2x+3k-2\right)}{x^{2}-6x+1}.\label{eq;hak}
\end{align}
\end{lemma}

\begin{proof}
Multiplying (\ref{eq;1b}) by $t^{i-1}$ and summing over $i=1,2,\ldots,k$, we obtain
\begin{align}
A_{k}(x,t)&=\frac{x(1-t^{k})}{1-t}+xt\left(A_{k}(x,t)-S_{k}(x,k,k)t^{k-1}\right)+xtB_{k}(x,t)+xtC_{k}(x,t)\nonumber\\&+xA_{k}(x,t)+xB_{k}(x,t)+xC_{k}(x,t)+\frac{x}{t}\left(A_{k}(x,t)-A_{k}(x,0)\right)\label{eq;20b}.
\end{align} Similarly, multiplying (\ref{eq;5b}) by $t^{i-1}$ and summing over $i=1,2,\ldots,k-1$, we obtain
\begin{align}
B_{k}(x,t)&=\frac{x(1-t^{k-1})}{1-t}+x\left(A_{k}(x,t)-S_{k}(x,k,k)t^{k-1}\right)+xB_{k}(x,t)+xC_{k}(x,t)\nonumber\\&+\frac{x}{t}\left(A_{k}(x,t)-A_{k}(x,0)\right)+\frac{x}{t}\left(B_{k}(x,t)-B_{k}(x,0)\right)\label{eq;25b}.
\end{align}
Finally, multiplying (\ref{eq;5c}) by $t^{i-1}$ and summing over $i=1,2,\ldots,k-1$, we obtain
\begin{align}
C_{k}(x,t)&=\frac{x(1-t^{k-1})}{1-t}+xt\left(C_{k}(x,t)-S_{k}(x,k-1,k)t^{k-2}\right)+x\left(A_{k}(x,t)-S_{k}(x,k,k)t^{k-1}\right)\nonumber\\&+xB_{k}(x,t)+xC_{k}(x,t)+\frac{x}{t}\left(A_{k}(x,t)-A_{k}(x,0)\right)\label{eq;5d}.
\end{align}
Taking $\lim_{t\to 1}$ in \eqref{eq;20b}, \eqref{eq;25b}, and \eqref{eq;5d}, we obtain a linear system of three equations in the variables $A_k(x,1),B_k(x,1)$, and $C_k(x,1)$. Solving this system gives \eqref{eq;10a}, \eqref{eq;10b}, and \eqref{eq;10c}.

To obtain \eqref{eq;hak}, notice that equation \eqref{eq1b} is equivalent to $S_k(x) = A_k(x,1) + B_k(x,1)+C_k(x,1)$ and the assertion follows from the previous part.
\end{proof}

\begin{theorem}\label{mth2}
We have
\begin{align*}
S_k(x)&=\frac{a_{1}(x,t_1,t_2)t_{1}^{k}t_{2}^{k}+a_{2}(x,t_1,t_2)t_{1}^{k}-a_2(x,t_2,t_1)t_{2}^{k}+a_{3}(x,t_1,t_2)}
{(t_1-1)(t_2-1)(x^2 - 6x + 1)(b_1(x,t_1,t_2)t_1^kt_2^k - b_2(x,t_1,t_2)t_1^k+b_2(x,t_2,t_1)t_2^k - b_3(x,t_1,t_2))},
\end{align*} where
{\footnotesize
\begin{align}
a_1(x,t_1,t_2)&=(t_2 - t_1)(((t_2 - 1)(t_1 - 1)k - 2t_2t_1)x^5\nonumber\\& + ((t_1 - 1)(1-t_2)(t_1^2t_2 + t_1t_2^2 + t_1^2 + t_1t_2 + t_2^2 + t_1 + t_2 + 3)k + 2t_1^3t_2^2 + 2t_1^2t_2^3 + 2t_2t_1 + 4t_1 + 4t_2)x^4\nonumber\\& + ((t_1 - 1)(t_2 - 1)(t_1^2t_2^2 + 4t_1^2t_2 + 4t_1t_2^2 + 3t_1^2 + 5t_1t_2 + 3t_2^2 + 4t_1 + 4t_2)k - 2t_2t_1(t_1^2t_2 + t_1t_2^2 + 3t_1t_2 + 1))x^3\nonumber\\& + ((t_1 - 1)(1-t_2)(3t_1^2t_2^2 + 3t_1^2t_2 + 3t_1t_2^2 + 7t_1t_2 + 3t_1 + 3t_2)k - 2t_1^3t_2^2 - 2t_1^2t_2^3 + 2t_1^2t_2^2 - 2t_2t_1 - 4t_1 - 4t_2)x^2\nonumber\\& + (3t_2t_1(t_2 - 1)(t_1 - 1)k + 2t_2t_1(t_1^2t_2 + t_1t_2^2 + 3t_1t_2 + 2))x - 2t_1^2t_2^2),\nonumber\\
a_2(x,t_1,t_2)&=(t_1t_2 - 1)((-t_2^2(t_2 - 1)(t_1 - 1)k - 2t_2^2t_1)x^5\nonumber\\& + ((t_1 - 1)(t_2 - 1)(t_1^2t_2^2 + t_1^2t_2 + t_1t_2^2 + t_1t_2 + 3t_2^2 + t_1 + t_2 + 1)k + 2t_1^3t_2 + 4t_1t_2^3 + 2t_2^2t_1 + 2t_1^2 + 4t_2^2)x^4\nonumber\\& + ((t_1 - 1)(1-t_2)(3t_1^2t_2^2 + 4t_1^2t_2 + 4t_1t_2^2 + t_1^2 + 5t_1t_2 + 4t_1 + 4t_2 + 3)k - 2t_1(t_1^2t_2 + 3t_1t_2 + t_2^2 + t_1))x^3\nonumber\\& + ((t_1 - 1)(t_2 - 1)(3t_1^2t_2 + 3t_1t_2^2 + 3t_1^2 + 7t_1t_2 + 3t_1 + 3t_2)k - 2t_1^3t_2 - 4t_1t_2^3 + 2t_1^2t_2 - 2t_2^2t_1 - 2t_1^2 - 4t_2^2)x^2\nonumber\\& + (3t_1t_2(t_1 - 1)(1-t_2)k + 2t_1(t_1^2t_2 + 3t_1t_2 + 2t_2^2 + t_1))x - 2t_1^2t_2),\nonumber\\
a_3(x,t_1,t_2)&=(t_1-t_2)((t_1^2t_2^2(t_1 - 1)(t_2 - 1)k - 2t_1^2t_2^2)x^5 \nonumber\\&+ ((t_1 - 1)(1-t_2)(3t_1^2t_2^2 + t_1^2t_2 + t_1t_2^2 + t_1^2 + t_1t_2 + t_2^2 + t_1 + t_2)k + 4t_1^3t_2^2 + 4t_1^2t_2^3 + 2t_1^2t_2^2 + 2t_1 + 2t_2)x^4\nonumber\\& + ((t_1 - 1)(t_2 - 1)(4t_1^2t_2 + 4t_1t_2^2 + 3t_1^2 + 5t_1t_2 + 3t_2^2 + 4t_1 + 4t_2 + 1)k - 2t_1^2t_2^2 - 6t_1t_2 - 2t_1 - 2t_2)x^3\nonumber\\& + ((t_1 - 1)(1-t_2)(3t_1^2t_2 + 3t_1t_2^2 + 7t_1t_2 + 3t_1 + 3t_2 + 3)k - 4t_1^3t_2^2 - 4t_1^2t_2^3 - 2t_1^2t_2^2 + 2t_1t_2 - 2t_1 - 2t_2)x^2\nonumber\\& + (3t_1t_2(t_1 - 1)(t_2 - 1)k + 4t_1^2t_2^2 + 6t_1t_2 + 2t_1 + 2t_2)x - 2t_1t_2),\nonumber\\
b_1(x, t_1, t_2) &= (t_1 - t_2)((-t_1^2t_2 - t_1t_2^2 - t_1^2 - t_1t_2 - t_2^2 - t_1 - t_2)x^2 + (t_1^2t_2^2 + t_1^2t_2 + t_1t_2^2 + 2t_1t_2 + t_1 + t_2)x - t_2t_1),\nonumber\\
b_2(x, t_1, t_2) &= (t_1t_2 - 1)((t_1^2t_2^2 + t_1^2t_2 + t_1t_2^2 + t_1t_2 + t_1 + t_2 + 1)x^2 + (-t_1^2t_2 - t_1t_2^2 - t_1^2 - 2t_1t_2 - t_1 - t_2)x + t_2t_1),\nonumber\\
b_3(x, t_1, t_2) &= (t_1 - t_2)((-t_1^2t_2 - t_1t_2^2 - t_1^2 - t_1t_2 - t_2^2 - t_1 - t_2)x^2 + (t_1^2t_2 + t_1t_2^2 + 2t_1t_2 + t_1 + t_2 + 1)x - t_2t_1),\nonumber
\end{align}} and
$$t_{1,2} = \frac{x(1-x)-x(1+x)\sqrt{x}}{2x^{2}}+\frac{\sqrt{x(1+x)(1-x)^2\pm2x(x^2-1)\sqrt{x}}}{2x\sqrt{x}}.$$
\end{theorem}

\begin{proof}
First, we solve \eqref{eq;5d} for $C_k(x,t)$ and substitute it into \eqref{eq;25b}. Then we solve the result for $B_k(x,t)$ and substitute it into \eqref{eq;20b}. This gives us the equation
\begin{align}
\frac{K(t)}{t}A_{k}(x,t)&=\frac{x(x^{2}t^{3}-xt^{2}-t+x)}{t}A_{k}(x,0)+x^{2}(t+1)(xt-1)B_{k}(x,0)\nonumber\\&-x^{2}t^{k-1}(t+1)(t-x)S_{k}(x,k-1,k)+xt^{k-1}(t^{3}x-t^{2}-xt+x^{2})S_{k}(x,k,k)\nonumber\\&+\frac{x(-xt^{k+2}+t^{k+1}+xt^{k}-x^{2}t^{k-1}+x^{2}t^{3}-xt^{2}-t+x)}{t-1},\label{eq;77b}
\end{align} where $K(t)=x^{2}t^{4}+2x(x-1)t^{3}+(1-3x+x^{2}-x^{3})t^{2}+2x(x-1)t+x^{2}$. The four roots of the kernel equation $K(t)=0$ are given by $t_1,1/t_1,t_2,$ and $1/t_2$. By substituting these four  roots into \eqref{eq;77b}, we obtain a linear system of four equations in the variables $A_k(x,0), B_k(x,0), S_k(x,k-1,k)$, and $S_k(x,k,k)$. Solving this system we obtain
\begin{align}
&A_k(x,0) = S_k(x,k,k) =\nonumber\\ &\frac{(x-1)(x+1)\left(f_1(x,t_1, t_2)t_1^kt_2^k+f_2(x,t_1, t_2)t_1^k-f_2(x,t_2, t_1)t_2^k+f_3(x,t_1, t_2)\right)}{(t_{1}-1)(t_{2}-1)\left(g_1(x,t_1,t_2)t_1^kt_2^k-g_2(x,t_1,t_2)t_1^k+g_2(x,t_2,t_1)t_2^k-g_3(x,t_1,t_2)\right)},\nonumber\\
&B_k(x,0) = S_k(x,k-1,k) =\nonumber\\
&\frac{q_1(x,t_1, t_2)t_1^kt_2^k +q_2(x,t_1, t_2)t_1^k -q_2(x,t_2, t_1)t_2^k +q_3(x,t_1, t_2)}{x(t_{1}-1)(t_{2}-1)\left(g_1(x,t_1,t_2)t_1^kt_2^k-g_2(x,t_1,t_2)t_1^k+g_2(x,t_2,t_1)t_2^k-g_3(x,t_1,t_2)\right)},
\end{align} where
{\footnotesize
\begin{align}
f_1(x,t_1, t_2) &= (t_2 - t_1)((-t_1 - t_2)x + t_2t_1(t_1t_2 + 1)),\nonumber\\
f_2(x,t_1, t_2) &= t_2(t_1t_2 - 1)(-t_2(t_1t_2 + 1)x + t_1(t_1 + t_2)),\nonumber\\
f_3(x,t_1, t_2) &=t_2t_1(t_2 - t_1)(t_1t_2(t_1 + t_2)x - t_1t_2 - 1),\nonumber\\
g_1(x,t_1,t_2) &= (t_1 - t_2)((-t_1^2t_2 - t_1t_2^2 - t_1^2 - t_1t_2 - t_2^2 - t_1 - t_2)x^2 + (t_1^2t_2^2 + t_1^2t_2 + t_1t_2^2 + 2t_1t_2 + t_1 + t_2)x - t_1t_2),\nonumber\\
g_2(x,t_1,t_2) &= ((t_1t_2 - 1)((t_1^2t_2^2 + t_1^2t_2 + t_1t_2^2 + t_1t_2 + t_1 + t_2 + 1)x^2 + (-t_1^2t_2 - t_1t_2^2 - t_1^2 - 2t_1t_2 - t_1 - t_2)x + t_1t_2),\nonumber\\g_3(x,t_1,t_2)&=(t_1 - t_2)((t_1^2t_2^2 + t_1^2t_2 + t_1t_2^2 + t_1t_2 + t_1 + t_2 + 1)x^2 + (-t_1^2t_2 - t_1t_2^2 - t_1^2 - 2t_1t_2 - t_1 - t_2)x + t_1t_2),\nonumber\\
q_1(x,t_1,t_2)&=(t_{1}-t_{2})(x^{2}+(t_{2}^{3}-t_{2})x-t_{2}^{2})(x^{2}+(t_{1}^{3}-t_{1})x-t_{1}^{2}), \nonumber\\
q_2(x,t_1,t_2)&=(t_{1}t_{2}-1)(t_{2}^{3}x^{2}+(1-t_{2}^{2})x-t_{2})(x^{2}+(t_{1}^{3}-t_{1})x-t_{1}^{2}), \nonumber\\
q_3(x,t_1,t_2)&=(t_{1}-t_{2})(t_{2}^{3}x^{2}+(1-t_{2}^{2})x-t_{2})(t_{1}^{3}x^{2}+(1-t_{1}^{2})x-t_{1}). \nonumber
\end{align}} Substituting these in \eqref{eq;hak}, the assertion follows.
\end{proof}

\begin{example}
    The generating function of the number of staircase $(RT_{2,n}, 3)$-words is given by $$\frac{x(x^{2}+5x+7)}{1-4x-4x^{2}-x^{3}}.$$
\end{example}

\begin{remark}
An alternative approach is the Transfer-matrix Method (e.g., \cite[Section 4.7]{St}), that we demonstrate now on the graph $P_2\times P_n$. Let $$\mathcal{S} = \{ij\in [k]^2\;:\; |i-j|\leq 1\},$$ ordered in some manner, say, lexicographically. Thus, $$\mathcal{S} = \{11, 12, 21, 22,23,\ldots,(k-1)k,kk\}.$$ Set $N=|\mathcal{S}|$. We construct an undirected graph $G$ whose vertex set is $V$. Two vertices $i_1j_1$ and $i_2j_2$ of $G$ are adjacent if $|i_1-i_2|\leq 1$ and $|j_1-j_2|\leq 1$. Let $A$ be the adjacency matrix of $G$. For example, if $k=3$ then $$A=\begin{pmatrix}
1&1&1&1&0&0&0\\
1&1&1&1&1&0&0\\
1&1&1&1&0&1&0\\
1&1&1&1&1&1&1\\
0&1&0&1&1&1&1\\
0&0&1&1&1&1&1\\
0&0&0&1&1&1&1
\end{pmatrix}.$$
We define $F_{ij}(x)=\sum_{n\geq0}(A^n)_{ij}x^n$. Clearly, $s_3(P_2\times P_n) = \sum_{1\leq i,j\leq N} (A^{n-1})_{ij}$. By \cite[Theorem 4.7.2]{St}, we have $$F_{ij}(x)=\frac{(-1)^{i+j}\det(I-x A\;:\;i,j)}{\det(I-x A)}.$$ Thus, the generating function of $s_3(P_2\times P_n)$ is given by  
\begin{align}
\sum_{n\geq 1}s_3(P_2\times P_n)x^n&=x\sum_{1\leq i,j\leq N} \sum_{n\geq 1}(A^{n-1})_{ij}x^{n-1}\nonumber\\&=x\sum_{1\leq i,j\leq N} F_{ij}(x)\nonumber\\&=\frac{x\sum_{1\leq i,j\leq N}(-1)^{i+j}\det(I-x A\;:\;i,j)}{\det(I-x A)}\nonumber\\&=\frac{-x^{6}-2x^{5}+9x^{4}+16x^{3}-15x^{2}-14x+7}{x^{7}+x^{6}-9x^{5}-9x^{4}+15x^{3}+7x^{2}-7x+1}\nonumber\\&=-\frac{x(x^{2}-7)}{x^{3}-x^{2}-5x+1}.
\end{align}

This may be done for every $k$ and for each of the three graph families that we study in this work. It follows that all the generating functions in this work are rational (see \cite{E} for a possible extension of this approach).
\end{remark}

\end{document}